\title{\Large \bf On the Stability of Continuous-Discontinuous Galerkin Methods for Advection-Diffusion-Reaction Problems} 
\date{\today}

\documentclass[11pt,leqno]{article}

\usepackage[paper=a4paper,dvips,top=3.0cm,left=3.0cm,right=3.0cm,
    foot=1cm,bottom=3.0cm]{geometry}\usepackage[dvips]{color}

\usepackage{authblk}

\author[1]{\normalsize Andrea Cangiani\thanks{andrea.cangiani@le.ac.uk}}
\author[2]{\normalsize John Chapman\thanks{john.chapman@durham.ac.uk}}
\author[1]{\normalsize Emmanuil H. Georgoulis\thanks{emmanuil.georgoulis@le.ac.uk}}
\author[2]{\normalsize Max Jensen\thanks{m.p.j.jensen@durham.ac.uk}}

\affil[1]{\small Department of Mathematics, University of Leicester, Leicester, United Kingdom}
\affil[2]{\small Department of Mathematics, University of Durham, Durham, United Kingdom}

\usepackage{graphicx}
\usepackage{subfigure}
\usepackage{psfrag}
\usepackage[dvips]{color}
\usepackage[usenames,dvipsnames]{xcolor}
\definecolor{ltgray}{gray}{0.95}

\usepackage[]{amsmath}

\usepackage{amsfonts}
\usepackage{amssymb}
\usepackage{textcomp}

\usepackage{amsthm}

\usepackage{amsopn}

\usepackage{xspace}

\DeclareMathAlphabet{\mathpzc}{OT1}{pzc}{m}{it}
\DeclareMathAlphabet{\mathcalligra}{T1}{calligra}{m}{n}

\usepackage{enumitem}

\usepackage{perpage}
\MakePerPage{footnote}

\usepackage[color,notref,final]{showkeys}
\definecolor{refkey}{rgb}{1,0,0}
\definecolor{labelkey}{rgb}{1,0,0}

\usepackage[pagewise]{lineno} 
\newcommand*\patchAmsMathEnvironmentForLineno[1]{%
  \expandafter\let\csname old#1\expandafter\endcsname\csname #1\endcsname
  \expandafter\let\csname oldend#1\expandafter\endcsname\csname end#1\endcsname
  \renewenvironment{#1}%
     {\linenomath\csname old#1\endcsname}%
     {\csname oldend#1\endcsname\endlinenomath}}%
\newcommand*\patchBothAmsMathEnvironmentsForLineno[1]{%
  \patchAmsMathEnvironmentForLineno{#1}%
  \patchAmsMathEnvironmentForLineno{#1*}}%
\AtBeginDocument{%
\patchBothAmsMathEnvironmentsForLineno{equation}%
\patchBothAmsMathEnvironmentsForLineno{align}%
\patchBothAmsMathEnvironmentsForLineno{flalign}%
\patchBothAmsMathEnvironmentsForLineno{alignat}%
\patchBothAmsMathEnvironmentsForLineno{gather}%
\patchBothAmsMathEnvironmentsForLineno{multline}%
}

\usepackage{listings}
\usepackage[font=small,labelfont=bf]{caption}

\newtheorem{theorem}[equation]{Theorem}
\newtheorem{definition}[equation]{Definition}
\newtheorem{lemma}[equation]{Lemma}

\newtheorem{assumption}[equation]{Assumption}

\newtheorem{remark}[equation]{Remark}

\numberwithin{equation}{section}
\numberwithin{figure}{section}
\numberwithin{table}{section}

\newcommand{\bdm}{\begin{displaymath}}
\newcommand{\edm}{\end{displaymath}}
\newcommand{\beq}{\begin{equation}}
\newcommand{\eeq}{\end{equation}}

\newcommand{\newVar}[2]{\newcommand{#1}{\ensuremath{#2}\xspace}}
\newVar\Naturals{\mathbb{N}}
\newVar\Integers{\mathbb{Z}}
\newVar\Rationals{\mathbb{Q}}
\newVar\Reals{\mathbb{R}}
\newVar\Complex{\mathbb{C}}
\newVar\Id{\mathbb{I}}

\newcommand{\norm}[1]{\ensuremath{\lVert#1\rVert}}

\newcommand{\triple}[1]{\ensuremath{|\hspace{-1pt}|\hspace{-1pt}| #1 |\hspace{-1pt}|\hspace{-1pt}|}}

\newcommand{\abs}[1]{\ensuremath{\lvert#1\rvert}}

\newcommand{\set}[1]{\ensuremath{\{#1\}}}

\newcommand{\Average}[1]{\ensuremath{\{ \! \! \{ #1 \}  \! \! \}} }
\newcommand{\average}[1]{\ensuremath{\Average{#1}}}

\usepackage{stmaryrd}
\newcommand{\Jump}[1]{\ensuremath{\llbracket #1 \rrbracket}}
\newcommand{\jump}[1]{\ensuremath{\Jump{#1}}}

\newcommand{\Order}[1]{{\ensuremath{\mathcal{O}(#1)}}}

\newVar\pqTree{{\cal P}}
\newVar\groups{\mathbf{G}}
\newVar\lastGroup{\ell}
\newVar{\currentId}{\ensuremath{\id}\xspace}

\renewcommand\epsilon{\varepsilon}
\renewcommand\phi{\varphi}
\renewcommand\theta{\vartheta}

\renewcommand{\exp}[1]{{\ensuremath{\mathrm{e}^{#1}}}}

\newcommand{\vectau}[0]{{\ensuremath{\boldsymbol{\tau}}}}
\newcommand{\vecb}[0]{{\ensuremath{b}}}
\newcommand{\vecx}[0]{{\ensuremath{\boldsymbol{x}}}}
\newcommand{\bdotn}[0]{{\ensuremath{\vecb \cdot n}}}

\newcommand{\Poly}[0]{{\ensuremath{\mathbb{P}}}}

\newcommand{\half}[0]{{\ensuremath{\frac{1}{2}}}}
\usepackage{units}
\newcommand{\fhalf}[0]{{\ensuremath{\nicefrac{1}{2}}}}

\newcommand{\nablah}[0]{{\ensuremath{\nabla_h}}}
\newcommand{\bdnab}[0]{{\ensuremath{\vecb\cdot\nablah}}}
\newcommand{\erf}[0]{{\ensuremath{\mathrm{Erf}}}}
\newcommand{\cb}[0]{{\ensuremath{r}}}
\newcommand{\cbhalf}[0]{{\ensuremath{r^{\fhalf}}}}

\newcommand{\cdG}[0]{{\ensuremath{\textrm{cdG}}}}
\newcommand{\dG}[0]{{\ensuremath{\textrm{dG}}}}
\newcommand{\cG}[0]{{\ensuremath{\textrm{cG}}}}
\providecommand{\C}[0]{}
\renewcommand{\C}[0]{{\ensuremath{\mathcal{C}}}}
\newcommand{\D}[0]{{\ensuremath{\mathcal{D}}}}
\newcommand{\VcG}[1]{{\ensuremath{V_\cG^{#1}}}}
\newcommand{\VdG}[1]{{\ensuremath{V_\dG^{#1}}}}
\newcommand{\VcdG}[1]{{\ensuremath{V_\cdG^{#1}}}}

\newcommand{\Bform}[0]{{\ensuremath{\mathcal{B}}}}

\newcommand{\Beps}[0]{{\ensuremath{\Bform_\epsilon}}}

\newcommand{\Btilde}[1]{{\ensuremath{\tilde{\Bform}_{#1}}}}
\newcommand{\vh}[0]{{\ensuremath{v_h}}}

\newcommand{\vhtilde}[0]{{\ensuremath{\tilde{v}_h}}}
\newcommand{\votilde}[0]{{\ensuremath{\tilde{v}_0}}}
\newcommand{\vepstilde}[0]{{\ensuremath{\tilde{v}_\epsilon}}}
\newcommand{\va}[0]{{\ensuremath{v_{\mathcal{A}}}}}
\newcommand{\vpitilde}[0]{{\ensuremath{\tilde{v}_\pi}}}
\newcommand{\SDG}[0]{{\ensuremath{S}}}
\newcommand{\SDGr}[0]{{\ensuremath{S}}}
\newcommand{\ws}[0]{{\ensuremath{v_\SDG}}}

\newcommand{\what}[0]{{\ensuremath{\hat{w}}}}
\newcommand{\vhat}[0]{{\ensuremath{\hat{v}}}}

\newcommand{\cont}[0]{\ensuremath{\cG}}
\newcommand{\discont}[0]{\ensuremath{\dG}}

\newcommand{\T}[0]{{\ensuremath{\mathcal{T}}}}
\newcommand{\Th}[0]{{\ensuremath{\mathcal{T}_h}}}

\newcommand{\TD}[0]{{\ensuremath{\mathcal{T}_{\discont}}}}
\newcommand{\TC}[0]{{\ensuremath{\mathcal{T}_{\cont}}}}

\newcommand{\EinTh}[0]{{\ensuremath{E \in \Th}}}

\newcommand{\Eho}[0]{\ensuremath{{\mathcal{E}_h^o}}}

\newcommand{\Eh}[0]{\ensuremath{{\mathcal{E}_h}}}
\newcommand{\EhC}[0]{\ensuremath{{\mathcal{E}_\cont}}}
\newcommand{\EhD}[0]{\ensuremath{{\mathcal{E}_\discont}}}

\newcommand{\GC}[0]{\ensuremath{{\Gamma_\cont}}}
\newcommand{\GD}[0]{\ensuremath{{\Gamma_\discont}}}

\newcommand{\insymbol}[0]{\ensuremath{\text{in}}}
\newcommand{\outsymbol}[0]{\ensuremath{\text{out}}}
\newcommand{\Gout}[0]{\ensuremath{\Gamma^{\outsymbol}}}
\newcommand{\Gin}[0]{\ensuremath{\Gamma^{\insymbol}}}

\newcommand{\bout}[0]{\ensuremath{\partial^{\outsymbol}}}
\newcommand{\bin}[0]{\ensuremath{\partial^{\insymbol}}}
\newcommand{\einEh}[0]{{\ensuremath{e \in \Eh}}}
\newcommand{\einEho}[0]{{\ensuremath{e \in \Eho}}}

\newcommand{\pE}[0]{{\ensuremath{\partial E}}}
\newcommand{\pO}[0]{{\ensuremath{\partial \Omega}}}
\newcommand{\OD}[0]{{\ensuremath{\Omega_\discont}}}
\newcommand{\OC}[0]{{\ensuremath{\Omega_\cont}}}

\newcommand{\patchE}[0]{{\ensuremath{\Delta_E}}}

\newcommand{\dt}[0]{{\ensuremath{\,\mathrm{d}t}}}
\newcommand{\dx}[0]{{\ensuremath{\,\mathrm{d}\vecx}}}
\newcommand{\ds}[0]{{\ensuremath{\,\mathrm{d}s}}}

\newcommand{\SZ}[1]{{\ensuremath{\mathcal{S\!Z}_h(#1)}}} 

\newcommand{\Proj}[2]{{\ensuremath{\Pi^{#1}_{#2}}}}

\renewcommand{\div}[0]{{\ensuremath{\mathrm{div}}}}

\newcommand{\SH}[2]{{\ensuremath{H^{#1}(#2)}}}
\newcommand{\SW}[3]{{\ensuremath{W^{#1,#2}(#3)}}}
\newcommand{\SL}[2]{{\ensuremath{L^{#1}(#2)}}}

\newcommand{\Cti}[0]{{\ensuremath{C_{\mathrm{ti}}}}}

\newcommand{\Lis}[0]{{\ensuremath{\Lambda_{\mathrm{is}}}}}

\newcommand{\Csz}[0]{{\ensuremath{C_{\mathrm{sz}}}}}

\renewcommand{\exp}[1]{{\ensuremath{\text{e}^{#1}}}}
\renewcommand{\div}[0]{{\ensuremath{\nabla \cdot \,}}}

\definecolor{dred}{RGB}{180,90,90}
\definecolor{dgreen}{RGB}{70,140,70}
\definecolor{dblue}{RGB}{100,100,180}
\definecolor{purple}{RGB}{139,0,204}

\begin{document}

\maketitle

%
\begin{abstract}
We consider a finite element method which couples the continuous Galerkin method away from internal and boundary layers with a discontinuous Galerkin method in the vicinity of layers. We prove that this consistent method is stable in the streamline diffusion norm if the convection field flows non-characteristically from the region of the continuous Galerkin to the region of the discontinuous Galerkin method. The stability properties of the coupled method are illustrated by numerical experiments.
\end{abstract}


%
\section{Introduction}  \label{sec:Introduction}

It is well known that the standard continuous Galerkin (cG) finite element method exhibits poor stability properties for singularly perturbed problems. In presence of sharp boundary or interior layers, non physical oscillations pollute the numerical approximation throughout the solution domain, and thus stabilisation techniques need to be employed; see~\cite{RST08} for a survey. The discontinuous Galerkin (dG) method offers a framework for the design of finite element methods with good stability properties; see, e.g., \cite{CKS00} for a survey of their development. This is achieved by relaxing the continuity requirements at the inter-element boundaries, where appropriate upwinded numerical fluxes can be employed. One such dG method is the Interior Penalty (IP) method considered herein.

However dG methods require more degrees of freedom (DOFs) compared to the standard cG method. In this work, we investigate if some of these additional DOFs can be removed without affecting the stability properties of the dG method. The motivation of this work is twofold: firstly we wish to improve our understanding of the well-established dG methods in response to the criticism in the increased number of DOFs; secondly we seek insight into the design of more efficient finite element methods, allowing to combine advantages the dG framework with smaller approximation spaces.

Conceptually we envisage in this work a finite element space that lies between the standard continuous and discontinuous Galerkin spaces. We construct such a space by applying standard continuous elements away from any boundary or internal layers (called the cG region) and discontinuous elements in the region of such layers (called the dG region). Therefore we call this method the \emph{continuous discontinuous Galerkin} (cdG) finite element method.

In the broader sense, methods of this type have been proposed before by Becker, Burman, Hansbo and Larson \cite{BBHL04}, Perugia and Sch\"{o}tzau \cite{PS01}, and Dawson and Proft \cite{DP02}. In \cite{BBHL04} a globally reduced discontinuous Galerkin method is studied, whose approximation space consists of the continuous piecewise linear functions enriched with piecewise constant functions. In \cite{PS01}, \cite{DP02} the local discontinuous Galerkin method \cite{CS98} has been used on the discontinuous region with transmission conditions on the subset of interelement boundaries where continuous and discontinuous elements meet. Our approach is different to \cite{PS01}, \cite{DP02} as we impose no such conditions beyond those already imposed by the dG method.
Removing the transmission conditions makes the approach more natural, but introduces some difficulties as possible over- and undershoots at the interface of the cG and dG regions must be controlled.
Indeed, the analysis is limited to the case where the cdG interface is non-characteristic so that jumps across the interface are controlled by the convection term. This allows us to derive a rigorous stabilty bound for the cdG approximation. On the cG region we split the numerical solution into an approximation to the hyperbolic problem which is weakly dependent on the diffusion coefficient $\epsilon$ and an approximation to the remaining part, which depends more strongly on $\epsilon$ but is small in size. Stability of the approximation on the dG region is shown using the approach of Buffa, Hughes and Sangali \cite{BHS06} with some extensions in the manner of Ayuso and Marini \cite{AM09}. To our knowledge this paper is the first to present a stability result for the proposed cdG method. 

Comparisons of our cdG method with various dG and cG methods have been undertaken by Cangiani, Georgoulis and Jensen \cite{CGJ06} and Devloo, Forti and Gomes \cite{DFG07}; see also \cite{CCGJ12c} which studies the cdG approximation as the limit of a dG approximation as the jump penalization on interelement boundaries tends to infinity.

The remainder of this work is organized as follows. In Section \ref{sec:formulation} we formulate the problem, introduce notation, and explain the assumptions we make. 
Control on the continuous region is considered in several stages in Section \ref{sec:epscontrol} and Section \ref{sec:hypcontrol}. Section \ref{sec:infsup} presents an inf-sup bound on the discontinuous region. The main result of the paper showing control independent of the perturbation parameter $\epsilon$ and mesh size $h$ follows in Section \ref{sec:combined}. The theoretical results are illustrated in Section \ref{sec:NumericalExperiments} by a numerical experiment.

%
\section{The Model Problem and Notation}  \label{sec:formulation}

Let $\Omega$ be a bounded Lipschitz domain in $\mathbb{R}^d$. We introduce the model advection-diffusion-reaction (ADR) problem
\begin{align}
  -\epsilon \Delta u + \vecb(\vecx) \cdot \nabla u + c(\vecx)u  = f(\vecx)& \qquad \text{for~} \vecx \in \Omega \subset \Reals^d ,\label{ADReqn}\\
  u  = 0  \!\!\!\qquad&\qquad \mathrm{on}~ \pO \label{ADReqnbc}
\end{align}
with constant diffusion coefficient $0<\epsilon \le \epsilon_\mathrm{max}$, $\vecb \in [W^{1,\infty}(\Omega)]^d$, $c \in \SL{\infty}{\Omega}$ and $f\in L^2(\Omega)$. For $0 < \epsilon \ll 1$ the solution to this problem typically exhibits boundary or interior layers.

Unless otherwise stated, we define $C$ throughout as a positive constant, independent of $\epsilon$ and the finite element approximation space and which may be redefined from line to line. $C$ may depend on $\epsilon_\mathrm{max}$. By $a \lesssim b$ we mean $a \le Cb$. 

We consider $\Th$ to be subdivision of $\Omega$ into non-overlapping shape regular simplices or hyper-cubes $E$, which we shall refer to as the triangulation. Denote by $\Eh$ the union of edges $e$ (or faces for $d\ge 3$) of the mesh and the union of internal edges by $\Eho$. Define $\Gamma$ as the union of boundary edges, i.e., those lying in $\pO$. The diameter of an element $E \in \Th$ is denoted $h_E$ and $h = \max_{\EinTh} h_E$. We also denote the mesh function by $h_E$, thus for $\vecx \in E$ we let $h_E(\vecx)$ be equal to the diameter of $E$. We only consider meshes where $h\le1$. Define $h_e := \min(h_{E^-}, h_{E^+})$ for $e \in \Eho$ with $e = \bar{E}^+ \cap \bar{E}^-$ for $E^+$, $E^- \in \Th$. The mesh is shape regular so there exists $C>0$ such that for all $E$ we have $h_e \le \half (h_{E^+} + h_{E^-} ) \le C h_e$, and $C\ge1$ such that $C^{-1}h_{E^-} \le h_{E^+} \le C h_{E^-}$.

We define by \emph{$\Omega$-decomposition} the splitting of $\Omega$ into two regions $\OC$ and $\OD$ such that for the closure $\overline{\Omega} = \overline{\OC \cup \OD}$, and we define by \emph{$\Th$-decomposition} the splitting of $\Th$ into two sub-meshes $\TC$ and $\TD$ such that $\TC \subset \OC$ and $\TD := \Th \setminus \TC$. By abuse of language, we denoted here by $\TC$ not just the sub-mesh but also the region it occupies. Define $\GC$ (resp.~$\GD$) to be the intersection of $\Gamma$ with $\overline{\T}_\cG$ (resp.~$\overline{\T}_\dG$). Define $J := \overline{\T}_\cG\cap\overline{\T}_\dG$ and by convention we say that the edges lying in $J$ are only part of the discontinuous Galerkin skeleton $\EhD$, the union of faces in $\overline{\T}_\dG$, and not part of the continuous Galerkin skeleton defined by $\EhC := \Eh \setminus \EhD$.

In Figure \ref{fig:cdGdecomposition} we illustrate a splitting for a problem where $\Omega=(0,1)^2$ and the solution exhibits layers at $x=1$ and $y=1$. The $\Omega$-decomposition is labelled, with the demarcation between the $\OC$ and $\OD$ regions given by a dashed line. A $\Th$-decomposition is shown with the $\TD$ region shaded and the edges in $J$ marked with a heavy line.

\begin{figure}[tb]
  \centering
  \includegraphics[width=0.50\textwidth]{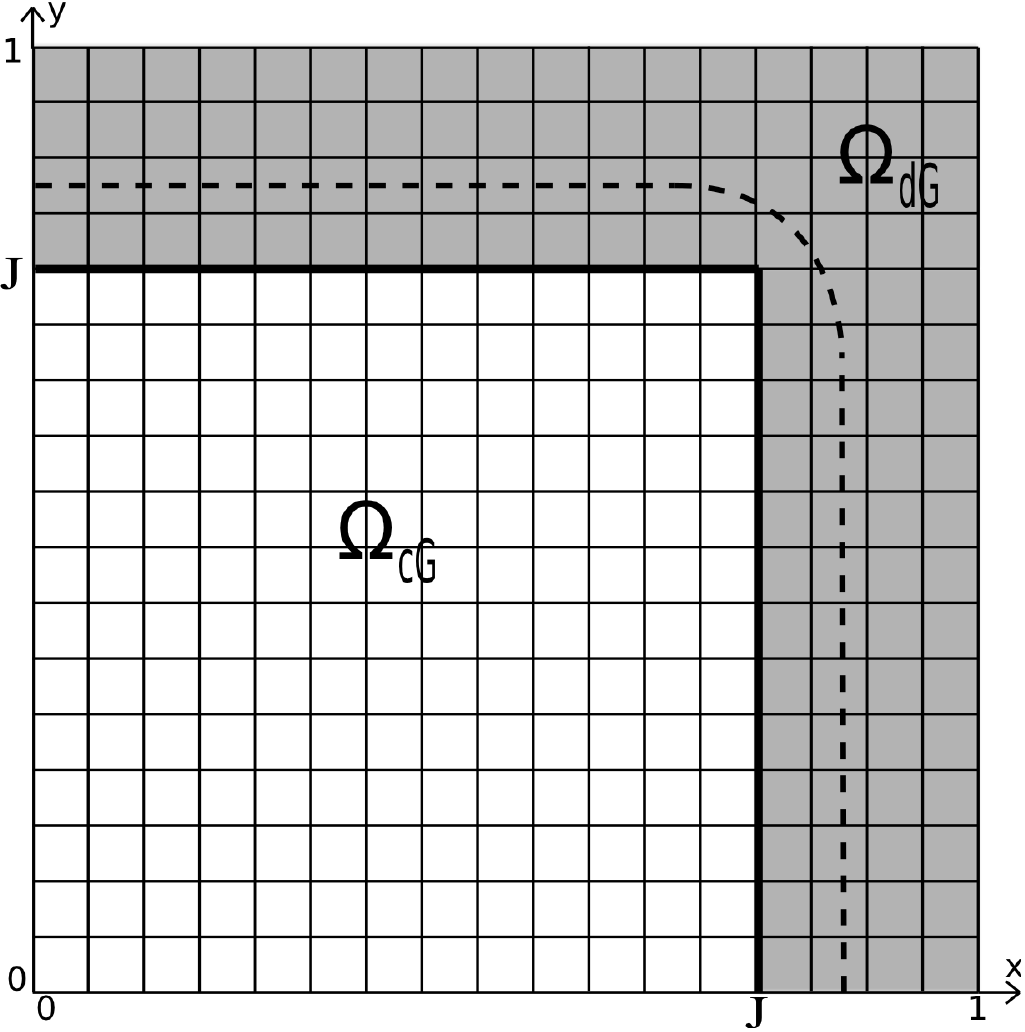}
  \caption{An example of a cdG decomposition.}
  \label{fig:cdGdecomposition}
\end{figure}

Given a generic scalar field $\nu: \Omega \to \Reals$, that may be discontinuous across an edge $e = \bar{E}^+ \cap \bar{E}^-$ for $E^+$, $E^- \in \Th$, we set $\nu^\pm := \nu|_{E^\pm}$, the interior trace on $E^\pm$ and similarly define $\vectau^\pm = \vectau|_{E^\pm}$ for a generic vector field $\vectau : \Omega \to \Reals^d$. Define the average and jump for a generic scalar as
\begin{equation*}
	\Average{\nu} := \frac{1}{2} (\nu^+ + \nu^-), \qquad \Jump{\nu} := \nu^+ n^+ + \nu^- n^-, \qquad \mathrm{on}~ e \in \Eho, 
\end{equation*}
and for a generic vector field as
\begin{equation*}
	\Average{\vectau} := \frac{1}{2} (\vectau^+ + \vectau^-), \qquad \Jump{\vectau} := \vectau^+ \cdot n^+ + \vectau^- \cdot n^-, \qquad \mathrm{on}~ e \in \Eho,
\end{equation*}
where $n^\pm$ is the unit outward pointing normal from $E^\pm$ on $e$. For $e \in \Gamma$ the definitions become
\begin{equation*}
	\average{\nu} := \nu, \qquad \jump{\nu} := \nu n, \qquad \average{\vectau} := \vectau, \qquad \mathrm{on}~ e \in \Gamma. 
\end{equation*}

We assume that the sign of $b \cdot n$ is the same for every $\vecx \in e$. Given a vector $b$ denote the inflow and outflow boundaries of $\Omega$ by
\begin{align*}
		\Gin := \set{\vecx \in \pO : b \cdot n \le 0}, \qquad \Gout :=  \set{\vecx \in \pO : b \cdot n > 0} 
\end{align*}
and for an element
\begin{align*}
		\bin E := \set{\vecx \in \pE : b \cdot n \le 0}, \qquad \bout E := \set{\vecx \in \pE : b \cdot n > 0}.
\end{align*}
On an edge $e$, we denote by $\nu^\insymbol$ the trace of a function $\nu$, taken from the element which contains $e$ in its inflow boundary. 
\begin{definition}
The cG space is defined by
\begin{equation*} \label{cGSpace}
  \VcG{} := \set{ v \in \SH{1}{\Omega} : \forall \EinTh, v|_E \in \Poly^k} 
\end{equation*}
and the dG space by
\begin{equation} \label{dGSpace}
  \VdG{} := \set{ v \in \SL{2}{\Omega} : \forall \EinTh, v|_E \in \Poly^k}   
\end{equation}
where $\Poly^{k}$ is the space of polynomials of either total (or on quadrilateral meshes partial) degree at most $k$ supported on $E$.
\end{definition}

We seek a finite element space $\VcdG{}$ that lies between the cG and dG spaces in the sense that
\begin{equation} \label{VcdGembedding}
  \VcG{} \subset \VcdG{} \subset \VdG{}.
\end{equation}
We construct this space by using continuous shape functions away from any boundary or internal layers present in the solution to \eqref{ADReqn} and discontinuous functions in the region of such layers. 
\begin{definition}
The cdG space is defined by
  \begin{equation} \label{cdGSpace}
  \begin{split}  
    \VcdG{} &:= \set{ v \in \SL{2}{\Omega} : \forall \EinTh, v|_E \in \Poly^k; v|_{\Gamma_\cG} = 0; v|_{\TC} \in C^0(\overline{\T}_{\!\cG})}.
  \end{split}
  \end{equation}
\end{definition}
Throughout we use the same  polynomial degree $k$ for $\VcG{}$, $\VdG{}$ and $\VcdG{}$.

Let $\chi$ be the characteristic function on $\TD$, i.e., that defined by
\begin{equation}\label{chidG}
  \chi := \left\{ \begin{array}{ll} 1 & \vecx \in \TD,\\ 0 & \vecx \in \TC. \end{array} \right.
\end{equation}
Then define
\begin{equation*} \label{cdGSpace-D}
  \VcdG{}(\TD) := \set{\chi v : v \in \VcdG{}}
\end{equation*}
and
\begin{equation*} \label{cdGSpace-C}
  \VcdG{}(\TC) := \set{(1-\chi) v : v \in \VcdG{}}.
\end{equation*}

In order to ensure a unique solution we make the following assumption.
\begin{assumption} \label{ass:ADRrho}
   We assume
  \begin{equation} \label{ADRrho}
    \cb(\vecx) := c(\vecx) - \half \div \vecb(\vecx) \ge \rho > 0 \qquad \forall \vecx \in \Omega,
  \end{equation}
   for some $\rho \in \Reals$.
\end{assumption}

\begin{definition} \label{def:Peclet}
We define the \emph{local mesh P\'{e}clet number} to be $\norm{\vecb}_\SL{\infty}{E} h_E / (2\epsilon)$, see \cite{RST08}.
\end{definition}

We consider meshes in the pre-asymptotic regime by making the following assumption.
\begin{assumption} \label{ass:Peclet}
 We assume that for $\epsilon = \epsilon_\mathrm{max}$ and every $\EinTh$ the local mesh P\'eclet number is greater than $\sqrt{h_E}$. Moreover, we require $\max \{\norm{h_E / b}_\SL{\infty}{\TD},h_E\} \le 1$.
\end{assumption}
As a consequence we have
\begin{equation} \label{Peclet1}
  \epsilon \le \epsilon_\mathrm{max} < \half \min_{\EinTh} h_E^{3/2} \norm{\vecb}_\SL{\infty}{\Omega}.
\end{equation}
This assumption, for a fixed $\vecb$, restricts the refinement of the triangulation for a given $\epsilon$. If we allowed $h \to 0$ for fixed $\epsilon > 0$ any layers would be resolved by the mesh and in the limit we would not see the non-physical oscillations associated with the cG approximation. We return to this question in Remark \ref{rem:cea}.


To characterise admissible $\Omega$-decompositions of the mesh we introduce the \emph{reduced} problem:
\begin{equation} \label{ADRreduced}
  \begin{split}
    \vecb \cdot \nabla u_0 + c u_0 &=f \quad \mathrm{on}~\Omega,\\
    u_0 &=0 \quad \mathrm{on}~\Gin.
  \end{split}
\end{equation}
Further, we define $u_\epsilon := u - u_0$, where $u$ is the solution to the ADR problem \eqref{ADReqn}.

The $\Omega$-decomposition is chosen such that $u_\epsilon$ and $u_0$ have additional regularity on $\OC$. In general we do not expect that $u_0 \in \SH{2}{\Omega}$, even if we place higher regularity requirements on $f$, see, e.g., \cite{BR82} and the references therein. 
\begin{assumption} \label{def:OC}
The set $\OC \subset \Omega$ is chosen such that $u_0 \in \SH{2}{\OC}$ and $\norm{u_\epsilon}_\SH{2}{\OC}$ is bounded uniformly in $\epsilon$, that is for every $0 < \epsilon \le \epsilon_\mathrm{max}$ 
  \begin{equation} \label{ass_ueps}
    \norm{u_\epsilon}_\SH{2}{\OC} \lesssim 1.
  \end{equation}
\end{assumption}

\subsection{Decoupled and Standard Formulations}

We define by $\nablah$ the elementwise gradient operator. We discretize the advection term by
\begin{equation} \label{ADRBa}
  \Bform_a(w,\what) := \sum_\EinTh \int_E (\vecb \cdot \nablah w) \what \dx - \sum_\einEho \int_e \vecb \cdot \jump{w} \, \what^\insymbol \ds - \sum_{e \in \Gin} \int_e (\bdotn) w \, \what \ds,
\end{equation}
and the reaction term by
\begin{equation} \label{ADRBr}
  \Bform_r(w,\what) := \sum_\EinTh \int_E c w \what \dx.
\end{equation}
The advection and reaction parts will frequently occur together and so for brevity we also define $\Bform_{ar}(w,\what) := \Bform_a(w,\what)+\Bform_r(w,\what)$.

For the diffusion term, besides the standard symmetric interior penalty method we also present a modified scheme, which we call the {\em decoupled} method. We refer to \cite{ABCM01} for a comprehensive study of the interior penalty method and alternative discretizations. For $w,\what \in \VdG{}$ the decoupled method is defined by
\begin{equation} \label{ADRBdtilde}
  \begin{split}
		\Btilde{d}(w,\what) &:=  \sum_\EinTh \int_E \nablah w \cdot \nablah \what  \dx \\
			& \qquad +\sum_{\einEh\setminus J} \int_e \sigma h_e^{-1} \jump{w} \cdot \jump{\what} - \left( \average{\nablah w}\cdot \jump{\what} + \average{\nablah \what}\cdot \jump{w}\right) \ds, 
  \end{split}
\end{equation}
yielding
\begin{equation} \label{ADRBepstilde}
  \Btilde{\epsilon} (w,\what) := \epsilon \Btilde{d}(w,\what) + \Bform_a(w,\what) + \Bform_r(w,\what).
\end{equation}
Here $\sigma>0$ is a discontinuity penalization parameter. 
With this formulation there is no control on the fluxes across $J$ (hence the name decoupled). 
We recall that the standard symmetric Interior Penalty (IP) method is given by
\begin{equation*}
  \begin{split}
		\Bform_d(w,\what) &:= \Btilde{d}(w,\what) + \sum_{e \in J} \int_e \sigma h_e^{-1} \jump{w} \cdot \jump{\what} - \left( \average{\nablah w} \cdot \jump{\what} + \average{\nablah \what} \cdot \jump{w}\right) \ds,
  \end{split}
\end{equation*}
and
\begin{equation*}
  \Bform_\epsilon (\what,w) := \epsilon \Bform_d(\what,w) + \Bform_a(\what,w) + \Bform_r(\what,w).
\end{equation*}
When restricted to the cdG space, the decoupled and standard IP forms become the bilinear form for the standard cG method on the continuous region. We first first analyse the stability properties of the decoupled formulation and infer stability for the standard IP method from a perturbation argument.

We introduce the following mesh dependent norm for $w \in \VdG{}$:
\begin{equation} \label{ADRtriple}
\begin{split}
  \triple{w}^2 &:= \norm{w}_d^2 + \norm{w}_{ar}^2,
\end{split}
\end{equation}
where
\begin{equation*}
  \norm{w}_d^2 := \sum_\EinTh \epsilon \abs{w}_\SH{1}{E}^2 + \sum_{\einEh} \epsilon\sigma h_e^{-1} \norm{\jump{w}}^2_\SL{2}{e},
\end{equation*}
and
\begin{equation*}
  \norm{w}^2_{ar} = \norm{\cbhalf w}^2_\SL{2}{\Omega} + \sum_\einEh \half\norm{\abs{\bdotn}^{\fhalf} \jump{w}}^2_\SL{2}{e},
\end{equation*}
with $\cb$ defined in \eqref{ADRrho}.

Recall that for the symmetric interior penalty method the parameter $\sigma$ is selected independently of the $\Th$ such that $\Bform_d$ is positive definite with a coercivity constant which is also independent of $\Th$.
We assume that $\sigma$ is such that for all $w \in \VdG{}$:
\begin{equation} \label{eqn:sigma}
\norm{\average{\nablah w}\cdot \jump{w}}_\SL{1}{\Eh} \le \half \norm{\nablah w}_\SL{2}{\Omega} \norm{\sqrt{\sigma / h_e} \jump{w}}_\SL{2}{\Eh}.
\end{equation}
Then, by Young's inequality, we have
\begin{equation*}
\half \norm{w}_d^2 \le \Bform_d(w,w).
\end{equation*}
We adopt for $\Btilde{d}$ the same $\sigma$ as for $\Bform_d$.

We introduce a projection operator following the presentation of \cite{AM09} which allows us to consider non-constant $\vecb$. For polynomial degree $k \ge 0$ consider the $L^2$-orthogonal projection $\Proj{}{\D}:\SL{2}{\Omega} \to \VcdG{}(\TD)$ defined by
\begin{equation} \label{ProjD}
	\int_\Omega \Proj{}{\D}(v) w \dx = \int_\Omega v w  \dx \qquad \forall w \in \VcdG{}(\TD).
\end{equation}
In particular $\Proj{}{\D}(v)|_\TC = 0$. Furthermore, for all elements $E$ of the mesh
\begin{equation} \label{DGL2-projection-norm}
	\norm{\Proj{}{\D}(v)}_\SL{2}{E} \le \norm{v}_\SL{2}{E} \qquad \forall v\in \SL{2}{E}.
\end{equation}

As $\Proj{}{\D}(v) \in \VcdG{}(\TD)$ we have for all $\EinTh$ the inverse inequality
\begin{align} \label{infsupInverse}
	\abs{\Proj{}{\D} (\vecb \cdot \nablah v)}_\SH{1}{E} &\lesssim h_E^{-1} \norm{\Proj{}{\D} (\vecb \cdot \nablah v)}_\SL{2}{E},
\end{align}
and using a trace inequality we have
\begin{equation} \label{infsupJump}
	\sum_{\einEh}  \norm{\jump{\Proj{}{\D} (\vecb \cdot \nablah v)}}^2_\SL{2}{e} \lesssim \sum_{\EinTh}  h_E^{-1} \norm{\Proj{}{\D}(\vecb \cdot \nablah v)}_\SL{2}{E}^2.
\end{equation}

Define the \emph{streamline} norm by
\begin{equation} \label{RestrictedSDGnorm}
  \norm{v}^2_\SDG := \triple{v}^2 + \sum_{\EinTh} \tau_E \norm{\Proj{}{\D}(\vecb\cdot\nablah v)}^2_\SL{2}{E},
\end{equation}
where $\tau_E$ is defined by
\begin{equation} \label{deftauCDG}
	\tau_E := \displaystyle\tau \min \left\{ \frac{h_E}{\norm{b}_\SL{\infty}{E}}, \frac{h_E^2}{\epsilon} \right\},
\end{equation}
and $\tau$ is a positive number at our disposal. 

\begin{definition} \label{def:vhtilde}
A decoupled cdG approximation to \eqref{ADReqn} is defined as $\vhtilde \in \VcdG{}$ satisfying
  \begin{equation} \label{decoupledcdG}
    \Btilde{\epsilon} (\vhtilde,v)  = \int_\Omega fv \dx \qquad \forall v \in \VcdG{}.
  \end{equation}
\end{definition}

\begin{definition} \label{def:vh}
A cdG approximation to \eqref{ADReqn} is defined as $\vh \in \VcdG{}$ satisfying
  \begin{equation}
    \Beps (\vh,v)  = \int_\Omega fv \dx \qquad \forall v \in \VcdG{}.
  \end{equation}
\end{definition}

We require that $b$ points on $J$ non-characteristically from $\TC$ to $\TD$.

\begin{assumption}\label{ass:binterface}
The $\Th$-decomposition is such that for every $e \in J$
\begin{equation} \label{binterface}
   \frac{1}{4} (\vecb(\vecx) \cdot n^\C)|_e > \epsilon_\mathrm{max} \frac{\sigma}{h_e^{3/2}} \qquad \forall \vecx \in e,
\end{equation}
where $n^\C$ represents the unit normal pointing from $\TC$ to $\TD$.
\end{assumption}

Observe that the scaling between $\epsilon$ and $h$ mirrors that of Assumption \ref{ass:Peclet}.

\begin{theorem} \label{thm:Bcoercivity}
On $\VdG{}$ the bilinear forms $\Btilde{\epsilon}$ and $\Beps$ are coercive with respect to $\triple{w}$:
  \begin{equation} \label{Bcoercivity}
    \frac{1}{4} \triple{w}^2 \le \Btilde{\epsilon}(w,w), \qquad \frac{1}{4} \triple{w}^2 \le \Beps(w,w), \qquad w \in \VdG{}.
  \end{equation}
\end{theorem}

\begin{proof}
For the advection and reaction terms using integration by parts we have
\begin{equation}\label{forJterms}
  \Bform_{ar}(w,w) = \norm{\cbhalf w}^2_\SL{2}{\Omega} + \sum_\einEh \int_e \half \abs{\bdotn}\jump{w}\cdot \jump{w} \ds.
\end{equation}
For the diffusion term it follows from Equation \ref{eqn:sigma} and Young's inequality that
\begin{align*}
\Btilde{d}(w,w) + \int_J \frac{\sigma}{h_e} \jump{w} \cdot \jump{w} \ds \ge \half \| w \|_d^2, \qquad \Bform_d(w,w) \ge \half \| w \|_d^2.
\end{align*}
Combing the last inequality with \eqref{forJterms}, the result now follows with Assumption \ref{ass:binterface}.
\end{proof}

It follows that $\vhtilde$ and $\vh$ exist and are unique. The following final assumption permits the use of an inverse inequality on the continuous Galerkin region. 
\begin{assumption} \label{ass:quasi}
The mesh $\TC$ is quasi-uniform.
\end{assumption}
It is convenient to denote the mesh-size on $\TC$ by $h_\TC = \| h_E \|_{L^\infty(\TC)}$.

The main result of this work is Theorem \ref{thm:vhstab}, which states that the cdG approximation is stable in the streamline diffusion norm whenever Assumptions \ref{ass:ADRrho}, \ref{ass:Peclet}, 
\ref{def:OC}, 
\ref{ass:binterface} and \ref{ass:quasi} are satisfied. 

In order to prove this result we establish first two separate stability bounds: We define $\vepstilde, \votilde \in \VcdG{}$ by the condition that for all $v\in\VcdG{}$
\begin{align}
  \Btilde{\epsilon}(\vepstilde,v) &= \Btilde{\epsilon}(u_\epsilon,v), \label{vepstilde}\\
  \Btilde{\epsilon}(\votilde,v) &= \Btilde{\epsilon}(u_0,v). \label{votilde}
\end{align}
Observe that by linearity of the decoupled cdG method we have $\vepstilde + \votilde = \vhtilde$. In Section \ref{sec:epscontrol} we derive a bound for the decoupled cdG approximation $\vepstilde$ to $u_\epsilon$ on $\TC$; in Section \ref{sec:hypcontrol} we obtain a bound for the decoupled cdG approximation $\votilde$ to $u_0$ on $\TC$. In Section \ref{sec:infsup} we establish an inf-sup condition with streamline control on $\TD$. Finally, in Section \ref{sec:combined} we combine these results to show stability of the decoupled and then of the standard cdG approximation.

%
\section{Bounds on the $\vepstilde$ Component on $\TC$}  \label{sec:epscontrol}

We introduce the projection operator of Scott and Zhang, \cite{SZ90} \cite[Section 1.6.2]{EG04}. 
\begin{lemma}[Scott-Zhang Projection] \label{thm:SZ}
The Scott-Zhang operator $\mathcal{S\!Z}_h : \SW{l}{p}{\Omega} \to \VcG{}$ is a mapping with the following properties: For $l > \half$ there exists a $\Csz>0$ such that for all $0\le m\le\min(1,l)$
  \begin{equation} \label{SZ1}
    \norm{\SZ{v}}_\SH{m}{\TC} \le \Csz \norm{v}_\SH{l}{\TC}\qquad \forall v \in \SH{l}{\TC},
  \end{equation}
  and provided $l\le k+1$ for all $E \in \TC$ and $0\le m \le l$ we have the approximation
  \begin{equation} \label{SZ2}
    \norm{v-\SZ{v}}_\SH{m}{E} \le \Csz h_E^{l-m} \abs{v}_\SH{l}{\patchE} \qquad \forall v \in \SH{l}{\patchE}.
  \end{equation}
  where $\patchE$ is the node patch of $E$, i.e., the set of cells in $\TC$ sharing at least one vertex with~$E$.
\end{lemma}

\begin{theorem} \label{thm:vepstilde}
The decoupled cdG approximation $\vepstilde$ is stable on the $\TC$ region in the sense that
  \begin{equation} \label{vepstildestab}
    \| \vepstilde \|_{H^1(\TC)} \lesssim 1.
  \end{equation}
\end{theorem}

\begin{proof}
We pick the \emph{auxiliary} solution $\va \in \VcdG{}$ as follows: On $\TC$, define $\va$ to be the Scott-Zhang projection of $u_\epsilon$; and on $\TD$ to be the dG approximation with boundary conditions given by $\SZ{u_\epsilon}$ on $J$ and $0$ on $\GD$, i.e.,
\begin{align*}
  \va &= \SZ{u_\epsilon} \quad\mathrm{on}~\TC,\\
  \Btilde{\epsilon}(\va,v) &= \Btilde{\epsilon} (u_\epsilon,v) \quad\forall v\in \VcdG{}(\TD).
\end{align*}
Set $\eta := u_\epsilon - \va$ and $\xi := \va - \vepstilde$, so $\eta + \xi = u_\epsilon - \vepstilde$. Notice that $\xi \in \VcdG{}$. The Galerkin orthogonality expressed by \eqref{vepstilde} and Theorem \ref{thm:Bcoercivity} give
\begin{equation} \label{xiorthog}
  {\textstyle \frac{1}{4}} \triple{\xi} \le \Btilde{\epsilon}(\xi,\xi) = -\Btilde{\epsilon}(\eta,\xi) = -\Btilde{\epsilon}(\eta,\xi-\chi\xi),
\end{equation}
where $\chi$ is defined in \eqref{chidG}. Note that $\xi - \chi\xi$ is continuous except on $J$ where $\jump{\xi-\chi\xi} = \xi^\C \cdot n^\C$ and $\average{\xi-\chi\xi} = \half \xi^\C$, where the superscript $\C$ indicates the trace taken from the continuous Galerkin side of $J$.

We examine each term of $\Btilde{\epsilon}$ in turn. For the diffusion parts we use Young's inequality
\begin{align*}
  -\Btilde{d}(\eta,\xi-\chi\xi) &\le 2 \abs{\eta}^2_\SH{1}{\TC} + {\textstyle \frac{1}{8}} \abs{\xi}^2_\SH{1}{\TC}.
\end{align*}
For the advection term we use Assumption \ref{ass:binterface} which ensures that flux terms on $J$ are zero as the upwind value of $\xi - \chi\xi$ vanishes. With Young's inequality we have
\begin{align*}
  -\Bform_{a}(\eta,\xi-\chi\xi) &\le \textstyle \frac{4}{\rho} \norm{b\cdot \nablah\eta}^2_\SL{2}{\TC} + \frac{\rho}{16}\norm{\xi}^2_\SL{2}{\TC},
\end{align*}
where $\rho$ is defined in \eqref{ADRrho}. Finally for the reaction term
\begin{align*}
  -\Bform_{r}(\eta,\xi-\chi\xi) &\le \textstyle \frac{4}{\rho} \norm{c}^2_\SL{\infty}{\Omega}\norm{\eta}^2_\SL{2}{\TC} + \frac{\rho}{16} \norm{\xi}^2_\SL{2}{\TC}.
\end{align*}
Using the previous three results, \eqref{xiorthog}, the definition of the norm \eqref{ADRtriple}, and Lemma \ref{thm:SZ} we gather $\xi$ terms on the left hand side to show, with $h_\TC = \| h_E \|_{L^\infty(\TC)}$,
\begin{align} \nonumber
  \textstyle \frac{1}{8} \triple{\xi}^2
  &\le \textstyle 2 \epsilon \abs{\eta}^2_\SH{1}{\TC} + \frac{4}{\rho}\norm{b\cdot \nablah\eta}^2_\SL{2}{\TC}+ \frac{4}{\rho} \norm{c}^2_\SL{\infty}{\Omega}\norm{\eta}^2_\SL{2}{\TC}\\ \label{vepstilde-bound}
  &\lesssim (\epsilon h_\TC^2 + h_\TC^2 + h_\TC^4) \norm{u_\epsilon}_\SH{2}{\OC}^2 \lesssim h_\TC^2
\end{align}
where in the final step we have used \eqref{ass_ueps}. As $\rho > 0$ we may use \eqref{vepstilde-bound} and an inverse inequality to show
\begin{equation} \label{vtildeepsinv}
  \norm{\xi}^2_\SH{1}{\TC} \lesssim h_\TC^{-2} \norm{\xi}^2_\SL{2}{\TC} \lesssim h_\TC^{-2} \triple{\xi}^2 \lesssim 1.
\end{equation}
Assumption \ref{def:OC} and \eqref{SZ1} give $\norm{\vepstilde}^2_\SH{1}{\TC} \lesssim 1$. 
\end{proof}

%
\section{Bounds on the $\votilde$ Component on $\TC$}  \label{sec:hypcontrol}

We now pick the auxiliary solution $\va$ as follows: On $\TD$ let it be $u_0$ and on $\TD$ be the dG approximation to $u_0$ with boundary conditions given by $u_0$ on $\GD \cup J$, i.e.,
\begin{align} \label{va_on_u0}
  \va = u_0 \qquad &\mathrm{on}~\TC,\\
  \Btilde{\epsilon}(\va,v) = \Btilde{\epsilon} (u_0,v) \qquad & \forall v\in \VcdG{}(\TD). \label{va_TD}
\end{align}

\begin{lemma} \label{thm:Bequivalenceva}
We have for all $v \in \VcdG{}$ that $\Btilde{\epsilon}(\va,v) = \Btilde{\epsilon}(\votilde,v)$.
\end{lemma}

\begin{proof}
Fix $v\in\VcdG{}$. Then using \eqref{votilde}
\begin{equation*}
  \Btilde{\epsilon}(\votilde,v) = \Btilde{\epsilon}(u_0,v) = \Btilde{\epsilon}(u_0,v-\chi v) +\Btilde{\epsilon}(u_0,\chi v)
\end{equation*}
where $\chi$ is defined in \eqref{chidG}. Observe that $\Btilde{\epsilon}(u_0,\chi v) = \Btilde{\epsilon}(\va, \chi v)$ by \eqref{va_TD}. Notice that $v - \chi v$ and $u_0$ are continuous on $\TC$. Recall that no integral over $J$ appears in the definition of $\Btilde{d}$. For $\Btilde{a}(\votilde,v-\chi v)$, the integral over $J$ vanishes since the value of $(v-\chi v)^\insymbol$ is zero because of Assumption \ref{ass:binterface}. Therefore $\Btilde{\epsilon}(\votilde,v-\chi v) = \Btilde{\epsilon}(u_0,v-\chi v) = \Btilde{\epsilon}(\va,v-\chi v)$.
\end{proof}

\begin{lemma} \label{thm:nablau0bound}
We have $\| \votilde \|_{H^1(\TC)} \lesssim 1$.
\end{lemma}

\begin{proof}
Define $\vpitilde$ to be
\begin{eqnarray*}
	\vpitilde &:= \left\{ 
		\begin{array}{ll}
			\SZ{u_0} & \mathrm{on}~ \TC,\\
			\va &\mathrm{on}~ \TD,
		\end{array} \right.
\end{eqnarray*}
and let $\eta:=\va-\vpitilde$, $\xi:=\vpitilde-\votilde$. With these definitions $\eta+\xi=\va -\votilde$, $\eta|_\TD = 0$ and $\xi$ and $\eta$ are continuous on $\TC$. Then using Lemma \ref{thm:Bequivalenceva} we have
\begin{align*}
  \textstyle \frac{1}{4} \triple{\xi}^2 \le \; & \Btilde{\epsilon}(\xi,\xi) = -\Btilde{\epsilon}(\eta,\xi)\\
  = \; &  - \int_\TC \epsilon \nablah \eta \cdot \nablah \xi + (\vecb\cdot\nablah\eta)\xi + c\eta\xi \dx + \int_J \vecb \cdot \jump{\eta} \xi^\insymbol \ds.
\end{align*}
Due to Assumption \ref{ass:binterface} we have $\xi^\insymbol = \xi^\D$, the trace from the dG side of $J$, and $\jump{\eta} = \eta^\C n^\C$, the trace and normal from the cG side of $J$. We split each of the terms using Young's inequality, giving
\begin{align} \nonumber
 \textstyle \frac{1}{4} \triple{\xi}^2 
  &\le 2 \epsilon \norm{\nablah \eta}^2_\SL{2}{\TC} + \frac{\epsilon}{8}\norm{\nablah \xi}^2_\SL{2}{\TC} + \frac{4}{\rho}\norm{\vecb \cdot \nablah \eta}_\SL{2}{\TC} + \frac{\rho}{16}\norm{\xi}^2_\SL{2}{\TC}\\
  &\qquad + \frac{4}{\rho}\norm{c}^2_\SL{\infty}{\Omega}\norm{\eta}^2_\SL{2}{\TC} + \frac{\rho}{16}\norm{\xi}^2_\SL{2}{\TC} + \int_J (\vecb \cdot n^\C \eta^\C)\xi^\D \ds.  \label{xi-bound-A}
\end{align}
For the final term we note that $\xi$ is a polynomial and so using Young's inequality and a trace and inverse inequality (with constant $\Cti$) gives
\begin{align} \label{xi-bound-B}
  \int_J (\vecb \cdot n^\C \eta^\C)\xi^\D \ds &\le \frac{4 \Cti\norm{\vecb}^2_\SL{\infty}{\Omega}}{h_e \rho} \norm{\eta^\C}^2_\SL{2}{J} +  \frac{\rho}{16} \norm{\xi}^2_\SL{2}{\TD}.
\end{align}
We combine \eqref{xi-bound-A} and \eqref{xi-bound-B} to hide all terms of $\xi$ under the norm on the left-hand side of \eqref{xi-bound-A}. Using \eqref{SZ2} for the terms of $\eta$ and a trace inequality gives
\begin{align*}
  \rho \norm{\xi}^2_\SL{2}{\TC} \le \triple{\xi}^2 \lesssim (\epsilon h_\TC^2 + h_\TC^4 + h_\TC^2)\norm{u_0}^2_\SH{2}{\TC} \lesssim h_\TC^2 \norm{u_0}^2_\SH{2}{\TC}
\end{align*}
and, by an inverse inequality, $\| \xi \|_{H^1(\TC)}^2 \lesssim 1$. Now the result follows from the stability of the Scott-Zhang operator.
\end{proof}

%
\section{Inf-Sup Condition}  \label{sec:infsup}

The following theorem is an adaptation of related stability bounds in \cite{BHS06} and \cite{AM09} to fit the above assumptions. Although the verification of the below inf-sup condition follows the overall structure in \cite{BHS06}, we state it here in detail as the present analysis extends the scope to non-constant advection coefficients via the incorporation of $\Proj{}{\D}$ as \cite{BHS06}. Moreover, it deals with the modification of the bilinear form on $J$ and it only has streamline control on the $\TD$ side. It is helpful to recall that $\Proj{}{\D} v|_\TC = 0$ for any $v$.

\begin{theorem} \label{thm:infsupcdGBtilde}
There exists a positive constant $\Lis$ which is independent of $h$ and $\epsilon$ but may depend on the polynomial degree, $\sigma$, and the constants in \eqref{infsupInverse} and \eqref{infsupJump} such that:
\begin{equation}
  \inf_{v \in \VcdG{}} \; \sup_{\vhat \in \VcdG{}} \; \frac{\Btilde{\epsilon}(v,\vhat)}{\norm{v}_\SDGr\norm{\vhat}_\SDGr} \ge \Lis.
\end{equation}
\end{theorem}

\begin{proof}
For an arbitrary $v \in \VcdG{}$, we define
\begin{equation} \label{infsupdecomposition}
  \vhat := v + \gamma \, \ws, \qquad \ws := \sum_{\EinTh} \tau_E \Proj{}{\D} (\bdnab v), 
\end{equation}
where $\gamma$ is a positive parameter at our disposal and $\tau_E$ is defined in \eqref{deftauCDG}. Note that through the definition of $\Proj{}{\D}$ we have $\vhat,\ws \in \VcdG{}$. Theorem \ref{thm:infsupcdGBtilde} is equivalent to showing the following two results:
\begin{align}
  \norm{\vhat}_\SDGr &\lesssim \norm{v}_\SDGr, \label{infsupcdG1}\\
  \Btilde{\epsilon}(v,\vhat) &\gtrsim \norm{v}^2_\SDGr \label{infsupcdG2}.
\end{align}

Consider first \eqref{infsupcdG1}. We examine each term of $\norm{\ws}^2_\SDGr$ in turn. We have
\begin{equation}\label{eq:ell_bound}
\begin{split}
\sum_\EinTh \epsilon\abs{\ws}^2_\SH{1}{E} \lesssim \; &  \sum_\EinTh \epsilon h_E^{-2}\norm{\tau_E \Proj{}{\D}(\bdnab v)}^2_\SL{2}{E}\\
\le \; & \sum_\EinTh \tau\tau_E \norm{\Proj{}{\D}(\bdnab v)}^2_\SL{2}{E} \lesssim \norm{v}_\SDGr^2.  
\end{split}
\end{equation}
Also
\begin{align} \label{eq:hyp_bound}
  \norm{\cbhalf \ws}^2_\SL{2}{\Omega} &\le \norm{\cb}_\SL{\infty}{\Omega} \sum_{E \in \Th} \tau_E^2 \norm{\Proj{}{\D}(\bdnab v)}^2_\SL{2}{E} \lesssim \norm{v}_\SDGr^2.
\end{align}
For the terms on the edges we use \eqref{infsupJump}. This gives
\begin{align} \label{eq:hyp_jump_bound}
  \sum_{\einEh} \norm{\abs{\bdotn}^\fhalf \jump{\ws}}^2_\SL{2}{e} &\lesssim \sum_\EinTh \norm{\vecb}_\SL{\infty}{\Omega} \tau_E^2 h_E^{-1} \norm{\Proj{}{\D}(\bdnab v)}^2_\SL{2}{E} \lesssim \norm{v}_\SDGr^2.
\end{align}
Similarly,
\begin{align} \label{eq:ell_jump_bound}
  \sum_{e \in \Eh} \frac{\sigma\epsilon}{h_e}\norm{\jump{\ws}}^2_\SL{2}{e} &\lesssim \sum_\EinTh \tau_E^2 \frac{\sigma\epsilon}{h_E^2}\norm{\Proj{}{\D}(\bdnab v)}^2_\SL{2}{E} \lesssim \norm{v}_\SDGr^2.
\end{align}
The final term of the streamline norm gives
\begin{align*}
	\sum_\EinTh \tau_E\norm{\Proj{}{\D}(\bdnab \ws) }^2_\SL{2}{E} &\le \sum_\EinTh \tau_E  \norm{\bdnab \left( \tau_E \Proj{}{\D} (\bdnab v) \right)}^2_\SL{2}{E}\\
&\lesssim  \sum_\EinTh \tau_E^3 \norm{\vecb}_\SL{\infty}{E}^2 h_E^{-2}\norm{\Proj{}{\D}(\bdnab v) }_\SL{2}{E}^2 \lesssim \norm{v}_\SDGr^2.
\end{align*}
Combining the above results we have $\norm{\ws}_\SDGr^2 \lesssim \norm{v}_\SDGr^2$. Using a triangle inequality we find
\begin{align*}
  \norm{\vhat}_\SDGr  &\le \norm{v}_\SDGr + \gamma \, \norm{\ws}_\SDGr \le C(\tau,\sigma,\gamma) \norm{v}_\SDGr,
\end{align*} 
which concludes the proof of \eqref{infsupcdG1}.

To prove \eqref{infsupcdG2} first consider the advection and reaction terms of the norm. Using the linearity of $\Bform_{ar}$ we have $\Bform_{ar}(v,\vhat) = \Bform_{ar}(v,v) + \gamma \Bform_{ar}(v,\ws)$. The second term is given by 
\begin{align*}
  \Bform_{ar}(v,\ws) &= \sum_\EinTh \int_E c v (\tau_E \Proj{}{\D}(\bdnab v)) + (\bdnab v)(\tau_E \Proj{}{\D}(\bdnab v)) \dx\\
  &\qquad - \sum_\einEho \int_e \vecb \cdot \jump{v}(\tau_E \Proj{}{\D} (\bdnab v))^\insymbol \ds - \sum_{e\in \Gin} \int_e (\bdotn) v(\tau_E \Proj{}{\D}(\bdnab v)) \ds.
\end{align*}
Using the properties of $\Proj{}{\D}$ given in \eqref{ProjD} the second term above becomes
\begin{equation}
\begin{split} \label{BarSpart}
  \sum_\EinTh \int_E (\bdnab v)(\tau_E \Proj{}{\D}(\bdnab v)) \dx & = \sum_\EinTh \int_E \tau_E \Proj{}{\D}(\bdnab v) \Proj{}{\D}(\bdnab v) \dx\\ &= \sum_\EinTh \tau_E \norm{\Proj{}{\D}(\bdnab v)}^2_\SL{2}{E}.
\end{split}
\end{equation}
Using Young's inequality we have
\begin{align*}
   \Bigl| \sum_\EinTh \int_E c v (\tau_E \Proj{}{\D}(\bdnab v)) \dx \Bigr| &\le \| c \|_{L^\infty(\Omega)} \sum_\EinTh \frac{1}{2} \norm{v}^2_\SL{2}{E} + \frac{1}{2}\tau_E^2 \norm{\Proj{}{\D}(\bdnab v)}^2_\SL{2}{E}
\end{align*}
and, where $C$ arises from a trace inequality and the number of edges per element,
\begin{align*}
  &- \sum_\einEho \int_e \vecb \cdot \jump{v}(\tau_E \Proj{}{\D} (\bdnab v))^\insymbol \ds - \sum_{e\in \Gin} \int_e (\bdotn) v(\tau_E \Proj{}{\D} (\bdnab v)) \ds\\
  & \qquad \le \sum_\einEh \frac{C \lambda}{2} \norm{\abs{\bdotn}^\fhalf\jump{v}}^2_\SL{2}{e} + \sum_\EinTh \frac{\tau_E \tau}{2\lambda} \norm{\Proj{}{\D}(\bdnab v)}^2_\SL{2}{E}.
\end{align*}
In conclusion, using \eqref{forJterms}, we have
\begin{equation} \label{infsupcdGBar}
\begin{split}
  \Bform_{ar}(v,\vhat) &\ge \left(\rho-\frac{\gamma \| c \|_{L^\infty(\Omega)}}{2}\right)\sum_\EinTh \norm{v}^2_\SL{2}{E} + \left(\half - \frac{\gamma C  \lambda}{2}\right)\sum_\einEh \norm{\abs{\bdotn}^\fhalf\jump{v}}^2_\SL{2}{e}\\
  &\qquad + \gamma\sum_\EinTh \left(\tau_E-\frac{\tau_E^2}{2}-\frac{\tau_E\tau}{2\lambda}\right) \norm{\Proj{}{\D}(\bdnab v)}^2_\SL{2}{\Omega}.
\end{split}
\end{equation}
Recall that $\norm{h_E / b}_\SL{\infty}{\TD} \le 1$ by Assumption \ref{ass:Peclet}, which imples $\tau_E \le 1$ for all $E \in \Th$. For general $v$, all terms on the right-hand side of \eqref{infsupcdGBar} are positive, provided $\lambda$ is large and $\gamma$ is small enough.

Equation (\ref{eqn:sigma}) ensures the continuity of $\Btilde{d}$ with respect to $\norm{\,\cdotp}_d$; thus 
\begin{align} \label{eq:cont}
\Btilde{d}(v,\vhat) \le C_1 \norm{v}_d \norm{\vhat}_d,
\end{align}
for some $C_1 > 0$. Recalling \eqref{eq:ell_bound} and \eqref{eq:ell_jump_bound}, it is clear that $\norm{\ws}_d \le C_2 \norm{v}_d$, for some $C_2 > 0$. Hence
\begin{align}
\Btilde{d}(v,\vhat) = \Btilde{d}(v,v) + \gamma \Btilde{d}(v,\ws) \ge \textstyle \frac{1}{4} \norm{v}_d^2 - \gamma C_1 \norm{v}_d \norm{\ws}_d.
\end{align}
Thus, if $\gamma < C_1 C_2 / 8$, then $\Btilde{d}(v,\vhat) \ge \frac{1}{8} \norm{v}_d^2$ which, combined with \eqref{infsupcdGBar}, gives \eqref{infsupcdG2}.
\end{proof}

%
\section{Stability of the Decoupled and Standard Approximations}  \label{sec:combined}

We saw that, under a set of suitable assumptions, the decoupled approximation satisfies the stability bounds:
\begin{align} \label{eq:vhtildestab}
\| \vhtilde \|_{H^1(\TC)} \lesssim 1, \qquad \| \vhtilde \|_\SDGr \lesssim \| f \|_{L^2(\Omega)}.
\end{align}
The first bound is a consequence of Theorem \ref{thm:vepstilde} and Lemma \ref{thm:nablau0bound}, the second of Theorem~\ref{thm:infsupcdGBtilde}. So while one has streamline-diffusion stability on $\TD$, an even stronger bound is available on~$\TC$ under the aforementioned assumptions. We now derive a stability result for the cdG method. We require that the geometry of the interface $J$ does not become significantly more complicated as the mesh is refined. More precisely, we require the boundedness of the trace operator.
\begin{theorem} \label{thm:vhstab}
Suppose that the operator norm of the trace $H^1(\TC) \to L^2(J)$ is bounded independently of $h$. Then, the cdG approximation $\vh$ is stable in the sense that
\[
h_\TC \norm{\nablah \vh}^2_\SL{2}{\TC} + \| \vh \|_\SDGr^2 \lesssim 1 + \| f \|_{L^2(\Omega)}^2.
\]
\end{theorem}

\begin{proof}
Set $\zeta:=\vh-\vhtilde$. Using the coercivity of $\Beps$, Galerkin orthogonality and the norm of the trace $H^1(\TC) \to L^2(J)$, we have
\begin{align*}
  \textstyle \frac{1}{4} \triple{\zeta}^2 &\le \Beps(\zeta,\zeta) = \Btilde{\epsilon}(\vhtilde,\zeta)-\Beps(\vhtilde,\zeta) + \Beps(\vh,\zeta)-\Btilde{\epsilon}(\vhtilde,\zeta)\\
  & =\Btilde{\epsilon}(\vhtilde,\zeta)-\Beps(\vhtilde,\zeta)\\
  &= \epsilon \int_J   \{ \nablah \vhtilde\} \cdot \jump{\zeta} + \{ \nablah \zeta \} \cdot \jump{\vhtilde} -\frac{\sigma}{h_e}  \jump{\vhtilde} \cdot \jump{\zeta}\ds\\
  &\lesssim \epsilon \Bigl( h_\TC \norm{\nablah \vhtilde}^2_\SL{2}{\Omega} \cdot \frac{\epsilon\sigma}{h_e^{3/2}} \norm{\jump{\zeta}}^2_\SL{2}{J} + \norm{\nablah \zeta}^2_\SL{2}{\Omega} \cdot \frac{\epsilon\sigma}{h_e^{1/2}} \norm{\jump{\vhtilde}}^2_\SL{2}{J}\\
& \qquad + \int_J \frac{\sigma}{h_e}  \jump{\vhtilde} \cdot \jump{\zeta} \ds \Bigr) \\
  &\lesssim \left( \epsilon \, h_\TC \norm{\nablah \vhtilde}^2_\SL{2}{\Omega} + \frac{\epsilon\sigma}{h_e^{1/2}} \norm{\jump{\vhtilde}}^2_\SL{2}{J} \right)^{\!\fhalf} \left( \epsilon \norm{\nablah \zeta}^2_\SL{2}{\Omega} + \frac{\epsilon\sigma}{h_e^{3/2}} \norm{\jump{\zeta}}^2_\SL{2}{J} \right)^{\!\fhalf}
\end{align*} 
and thus, using Assumption \ref{ass:binterface} for $\epsilon\sigma h_e^{-3/2}$,
\begin{align} \label{zetabound}
  \triple{\zeta}^2 \lesssim \epsilon h_\TC \norm{\nablah \vhtilde}^2_\SL{2}{\Omega} + \frac{\epsilon\sigma}{h_\TC^{1/2}} \norm{\jump{\vhtilde}}^2_\SL{2}{J}.
\end{align}
Dividing through by $h_\TC$ and using an inverse inequality on $\rho \norm{\zeta}_\SL{2}{E}$ gives
\begin{align} \label{hzetabound}
  h_\TC \norm{\nablah \zeta}^2_\SL{2}{\TC} \lesssim h_\TC^{-1} \triple{\zeta}^2 \lesssim \epsilon \norm{\nablah \vhtilde}^2_\SL{2}{\Omega} + \frac{\epsilon\sigma}{h_\TC^{3/2}}\norm{\jump{\vhtilde}}^2_\SL{2}{J}.
\end{align}
Using Assumptions \ref{ass:Peclet} and \ref{ass:binterface}, as well as \eqref{eq:vhtildestab}, we bound each of the terms in \eqref{hzetabound}. Using a triangle inequality on $\| \nabla \zeta \|_{L^2(\TC)}$ we conclude that 
\[
h_\TC \norm{\nablah \vh}^2_\SL{2}{\TC} \lesssim 1.
\]
To show that $\| \vh \|_\SDGr$ is bounded we establish an inf-sup condition for $\Beps$. Indeed, \eqref{infsupcdG1} may be used without change. It remains to transfer \eqref{infsupcdG2} to $\Bform_{\epsilon}$. The inequality \eqref{infsupcdGBar} is still available as the discretesation of the lower-order terms did not change. Different is that we now use $\Bform_{d}(v,\vhat) \le C_1 \triple{v} \cdot \triple{\vhat}$ in place of \eqref{eq:cont}, justified by Assumption \ref{ass:binterface}. Appealing to \eqref{eq:ell_bound}--\eqref{eq:ell_jump_bound}, one has $\triple{\ws} \le C_2 \triple{v}$ for some $C_2 > 0$. Hence
\begin{align}
\Btilde{d}(v,\vhat) = \Btilde{d}(v,v) + \gamma \, \Btilde{d}(v,\ws) \ge \textstyle \frac{1}{4} \norm{v}_d^2 - \gamma \, C_1 \triple{v} \triple{\ws}.
\end{align}
For $\gamma \, C_1 C_2$ small enough and $\lambda$ sufficiently large, $\gamma \, C_1 \triple{v} \triple{\ws}$ is bounded by $\frac{1}{8} \norm{v}_d^2 + \half \Bform_{ar}(v,\vhat)$, using again the positivity of the terms in \eqref{infsupcdGBar}.
\end{proof}


\begin{remark} \label{rem:cea}
Due to Assumptions \ref{ass:Peclet} and \ref{ass:binterface} the above stability bound is valid for the regime $\epsilon \lesssim h_E^{3/2} \norm{\vecb}_\SL{\infty}{\Omega}$. For completeness we briefly outline here how $\sqrt{h} \| \cdot \|_{H^1(\Omega)}$ stability of the cdG method is established if $\epsilon \gtrsim h_E^{3/2} \norm{\vecb}_\SL{\infty}{\Omega}$. The stability proof is in this case easier because the mesh P\'{e}clet number is smaller. If $\Omega$ is smooth or convex and the coefficients have sufficient regularity then $u$ is in $H^2(\Omega)$. Indeed for $\epsilon \ge \half h_E^{3/2} \norm{\vecb}_\SL{\infty}{\Omega}$ the $H^2$ norm of $u$ is uniformly bounded in $\epsilon$. Suppose that the mesh is quasi-uniform. By Ce\`a's Lemma, with $h := \max_E h_E$, a standard argument gives
\begin{align*}
\| u - \vh \|_{H^1(\Omega)} \lesssim (1 + \epsilon^{-1}) h \| u \|_{H^2(\Omega)} \lesssim (h + h^{-1/2}) \| u \|_{H^2(\Omega)},
\end{align*}
and thus $\sqrt{h} \| u - \vh \|_{H^1(\Omega)} \lesssim 1$.
\end{remark}

%
\section{Numerical Experiment} \label{sec:NumericalExperiments}

\begin{figure}[tb]
  \centering
  \subfigure[Solution $u$ given by \eqref{cdG3true}.]{\includegraphics[width=0.40\textwidth]{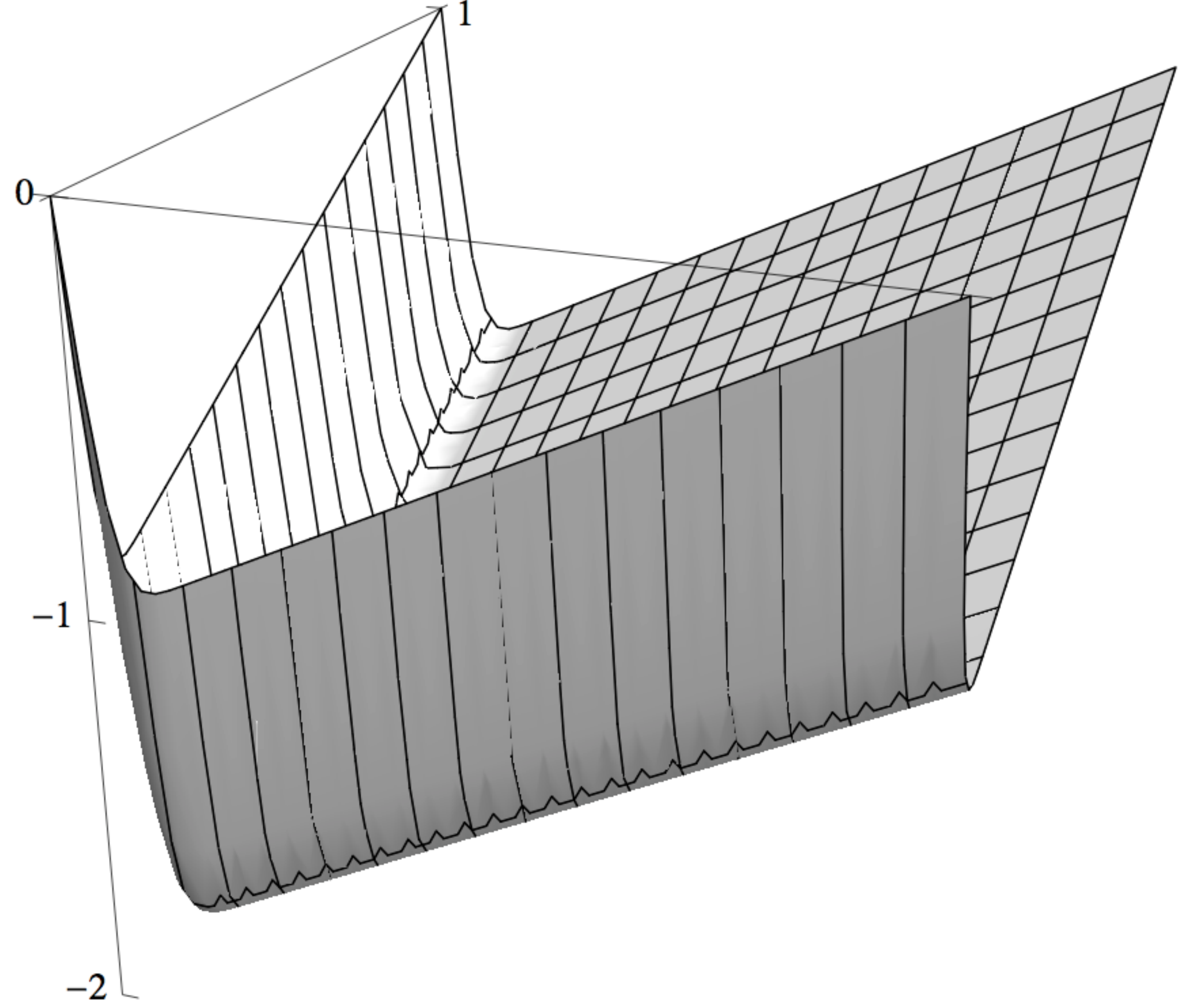}} \hspace{1.5cm}
  \subfigure[Solution $u_\epsilon$ given by \eqref{cdG3ueps}]{\includegraphics[width=0.40\textwidth]{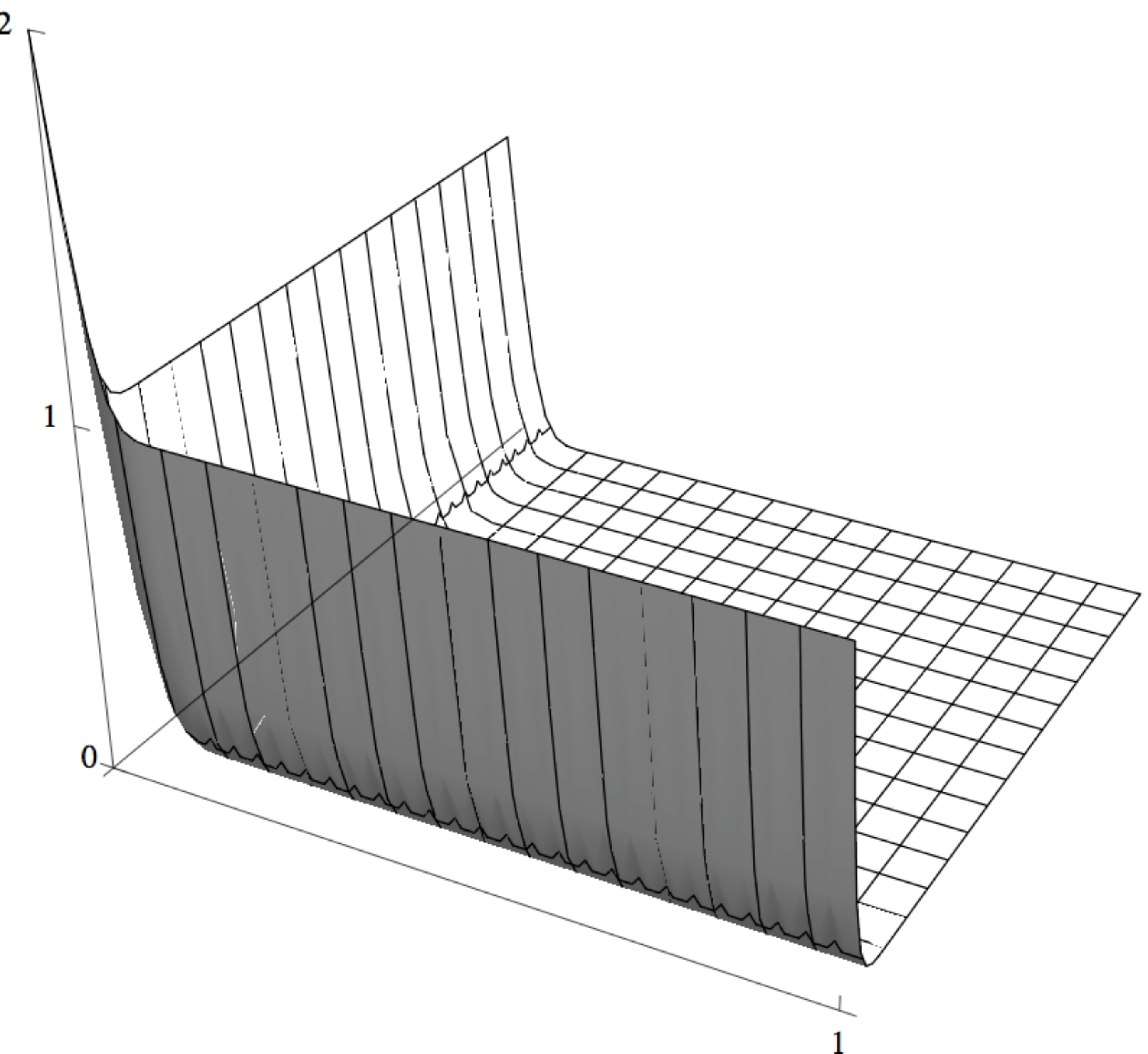}}  
  \caption{Solution $u$ and $u_\epsilon$ for $\epsilon = 10^{-3}$.}
  \label{fig:cdG3true}
\end{figure}

Let $\Omega = (0,1)^2$. We seek to solve
\begin{equation} \label{cdG3}
  -\epsilon \Delta u + (-x,-y)\cdot \nabla u = -x-y
\end{equation}
with Dirichlet boundary conditions chosen such that the solution is given by
\begin{equation} \label{cdG3true}
  u(x,y) = x+y - \frac{ \erf \left(x/\sqrt{2\epsilon}\right) +\erf \left(y/\sqrt{2\epsilon}\right)}{\erf \left(1/\sqrt{2\epsilon}\right)}
\end{equation}
where $\erf$ is the error function defined by
$
	\erf(x) = \frac{2}{\sqrt{\pi}} \int_0^x \exp{-t^2} \dt
$. 
For $0<\epsilon \ll 1$ this problem exhibits an exponential boundary layer along the outflow boundaries $x=0$ and $y=0$ of width $\Order{\sqrt{\epsilon}}$.

Away from the layers the boundary conditions on the inflow boundaries $x=1$ and $y=1$ are well approximated by $y-1$ and $x-1$ respectively. The hyperbolic solution with these boundary conditions is given by $u_0(x,y) = x+y-2$. This gives
\begin{equation} \label{cdG3ueps}
  u_\epsilon(x,y) = 2-\frac{ \erf \left(x/\sqrt{2\epsilon}\right) +\erf \left(y/\sqrt{2\epsilon}\right)}{\erf \left(1/\sqrt{2\epsilon}\right)}.
\end{equation}

We plot \eqref{cdG3true} and \eqref{cdG3ueps} for $\epsilon = 10^{-3}$ in Figure \ref{fig:cdG3true}; note that away from the layers the solution $u_\epsilon$ is close to zero. 

We let $\OC$ be a set of the type $(1-\delta,1)^2$, $0 < \delta < 1$. For each $\delta$ the supremum 
\[
\sup_{\epsilon \in (0,\epsilon_\mathrm{max}]} \norm{u_\epsilon}_\SH{2}{\OC}
\]
is finite. The dependence of $\norm{u_\epsilon}_\SH{2}{\OC}$ with respect to $\delta$ and to $\epsilon$ is illustrated in Figure \ref{fig:Ex3H2norm}.

\begin{figure}[tb] 
  \centering
  \includegraphics[height=6.5cm]{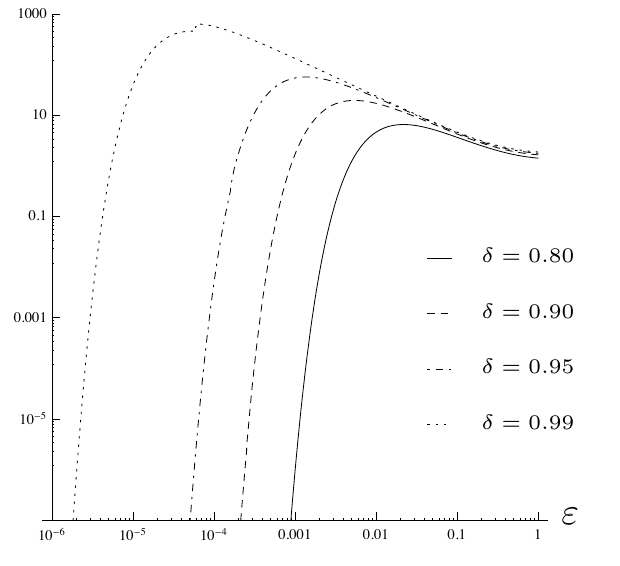}
  \caption{$\norm{u_\epsilon}_\SH{2}{\OC}$ for different values of $\delta$.}
  \label{fig:Ex3H2norm}
\end{figure}

\begin{figure}[tb]
  \centering
  \includegraphics[height=6.5cm]{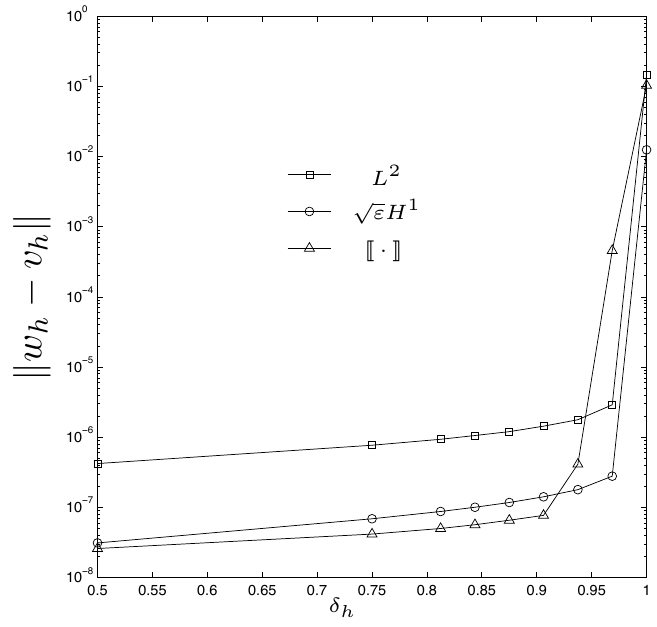}
  \caption{Difference between cdG and dG solutions.}
  \label{fig:Ex3diff}
\end{figure}

For this example $c - \half \nabla \cdot \vecb = 1$, so Assumption \ref{ass:ADRrho} is satisfied. Further, we fix $\epsilon = 10^{-6}$ and consider uniform square meshes of edge length $2^{-5}$ so that Assumption \ref{ass:Peclet} is also satisfied as the smallest local mesh P\'eclet number being $488.28$.

We define $\TC = [1-\delta_h,1]^2$, where $\delta_h = m 2^{-5}$, $m \in \set{0,\ldots,32}$. Note that, having fixed the mesh, $\delta_h$ is a discrete parameter. The interface $J$ is composed of the edges lying on the lines $y=\delta_h$ for $x\ge\delta_h$ and $x=\delta_h$ for $y \ge \delta_h$. The smallest value of $\bdotn$ is $\delta_h$ occuring on the edges containing the point $(\delta_h,\delta_h)$. Thus, in this case, Assumption \ref{ass:binterface} reads $\frac{1}{4}mh>\epsilon\frac{\sigma}{h^{3/2}}$ and  is satisfied for all $m \in \set{1,\ldots,32}$ (note that this assumption is trivially satisfied when $m=0$) for this choice of $\epsilon$ and $h$ if $\sigma<200$; in the shown computations $\sigma = 10$. 

In Figure \ref{fig:Ex3diff} we plot the $\SL{2}{\Omega}$ norm, $\sqrt{\epsilon}$ weighted $\SH{1}{\Th}$ semi-norm, and $L^2$ norm of the jumps on $\Eh$(represented by $\jump{\,\cdotp}$) for both the difference in the dG and cdG approximations and the error in the cdG approximation. Note from Figure \ref{fig:Ex3diff} that the difference in the approximations increases only very slowly until the final data points (where $\TC \approx \Th$). When the continuous region covers the layer, non-physical oscillations pollute the approximation in function of $\delta_h$. 

In Table \ref{table:Ex3dofs} we show the number of degrees of freedom (dofs) as the continuous region is increased. Reducing the degrees of freedom to approximately 30\% of the dG method 
degrees of freedom results in only a very slight difference in the norm, thus showing that a considerable saving can be made without compromising stability.
\begin{table}[htb] 
\centering
\begin{tabular}{|c|c|c|c|}
  \hline
  $1-\delta$ &  dofs & \% of dG dofs & $\sqrt{\epsilon}\norm{\nablah(w_h - v_h)}_\SH{1}{\Omega}$ \\
  \hline
  \hline
  dG & 4096 & 100 & 0.0\\
  $8\times2^{-5}$ & 3361 & 82.1 & 3.1157e-08\\
  $16\times2^{-5}$ & 2417 & 59.0 & 6.8911e-08\\
  $24\times2^{-5}$ & 2121 & 51.8 & 8.7544e-08\\
  $30\times2^{-5}$ & 1457 & 35.6 & 1.7934e-07\\  
  $31\times2^{-5}$ & 1276 & 31.2 & 2.7896e-07\\
  cG & 1089 & 26.6 & 1.2444e-02\\
  \hline 
\end{tabular}
\caption{Degrees of freedom with $\epsilon = 10^{-6}$.}\label{table:Ex3dofs}
\end{table}

We finally remark that, at least for the example considered here, the choice of $\TC$ leaving one layer of elements at the outflow boundary is optimal. Indeed, adding even a single element to the $\TC$ region results in oscillations polluting the solution. For example, for 
\begin{equation*}
\TC = [2^{-5},1]^2 \cup ([0.5,0.5+2^{-5}]\times [0,2^{-5}]),
\end{equation*}
i.e., adding a sinlge element to $\TC$ halfway along the $x$-axis, results in
\begin{align*}
\norm{w_h - v_h}_\SL{2}{\Omega} &= 4.7966\times10^{-2},\\
\sqrt{\epsilon}\norm{\nablah(w_h - v_h)}_\SL{2}{\Omega} &= 4.5008\times10^{-3},
\end{align*}
a significant increase on the norms for $\TC = [2^{-5},1]^2$. Notice that this choice of $\TC$ violates Assumption \ref{ass:binterface}.

%

\section{Acknowledgements:} We gratefully thank the Archimedes Center for Modeling, Analysis and Computation in Crete for hosting the authors during the preparation of this manuscript.

\small


\begin{thebibliography}{10}

\bibitem{ABCM01}
{\sc D.~N. Arnold, F.~Brezzi, B.~Cockburn, and L.~D. Marini}, {\em Unified
  analysis of discontinuous {G}alerkin methods for elliptic problems}, SIAM
  Journal on Numerical Analysis, 39 (2001), pp.~1749--1779.

\bibitem{AM09}
{\sc B.~Ayuso and L.~D. Marini}, {\em Discontinuous {Galerkin} methods for
  advection-diffusion-reaction problems}, SIAM Journal on Numerical Analysis,
  47 (2009), pp.~1391--1420.

\bibitem{BR82}
{\sc C.~Bardos and J.~Rauch}, {\em Maximal positive boundary value problems as
  limits of singular perturbation problems}, Transactions of the American
  Mathematical Society, 270 (1982), pp.~pp. 377--408.

\bibitem{BBHL04}
{\sc R.~Becker, E.~Burman, P.~Hansbo, and M.~Larson}, {\em A reduced
  {P}1-discontinuous {G}alerkin method}, tech. rep., EPFL, 2004.
\newblock EPFL-IACS report 05.2004.

\bibitem{BHS06}
{\sc A.~Buffa, T.~J.~R. Hughes, and G.~Sangalli}, {\em Analysis of a multiscale
  discontinuous {Galerkin} method for convection-diffusion problems}, SIAM
  Journal on Numerical Analysis, 44 (2006), pp.~1420--1440.

\bibitem{CCGJ12c}
{\sc A.~Cangiani, J.~Chapman, E.~H. Georgoulis, and M.~Jensen}, {\em On local
  super-penalization of interior penalty {G}alerkin methods}, submitted jounral
  article,  (2012).

\bibitem{CGJ06}
{\sc A.~Cangiani, E.~H. Georgoulis, and M.~Jensen}, {\em Continuous and
  discontinuous finite element methods for convection-diffusion problems: A
  comparison}, in International Conference on Boundary and Interior Layers,
  G\"{o}ttingen, July 2006.

\bibitem{CKS00}
{\sc B.~Cockburn, G.~Karniadakis, and C.~Shu}, {\em The development of
  discontinuous {G}alerkin methods}, in Discontinuous Galerkin Methods: Theory,
  Computation, and Applications, B.~Cockburn, G.~Karniadakis, and C.~Shu, eds.,
  vol.~11 of Lecture Notes in Computational Science and Engineering, Springer,
  2000.

\bibitem{CS98}
{\sc B.~Cockburn and C.-W. Shu}, {\em The local discontinuous {G}alerkin method
  for time-dependent convection-diffusion systems}, SIAM Journal on Numerical
  Analysis, 35 (1998), pp.~2440--2463.

\bibitem{DP02}
{\sc C.~Dawson and J.~Proft}, {\em Coupling of continuous and discontinuous
  {G}alerkin methods for transport problems}, Computer Methods in Applied
  Mechanics and Engineering, 191 (2002), pp.~3213 -- 3231.

\bibitem{DFG07}
{\sc P.~R.~B. Devloo, T.~Forti, and S.~M. Gomes}, {\em A combined
  continuous-discontinuous finite element method for convection-diffusion
  problems}, Latin American Journal of Solids and Structures, 2 (2007),
  pp.~229--246.

\bibitem{EG04}
{\sc A.~Ern and J.-L. Guermond}, {\em Theory and Practice of Finite Elements},
  Springer-Verlag, New York, 2004.

\bibitem{PS01}
{\sc I.~Perugia and D.~Sch\"otzau}, {\em On the coupling of local discontinuous
  {G}alerkin and conforming finite element methods}, Journal of Scientific
  Computing, 16 (2001), pp.~411--433.

\bibitem{RST08}
{\sc H.-G. Roos, M.~Stynes, and L.~Tobiska}, {\em Numerical Methods for
  Singularly Perturbed Differential Equations: Convection-Diffusion and Flow
  Problems}, Springer-Verlag, Berlin, {S}econd~ed., 2008.

\bibitem{SZ90}
{\sc L.~R. Scott and S.~Zhang}, {\em Finite element interpolation of nonsmooth
  functions satisfying boundary conditions}, Mathematics of Computation, 54
  (1990), pp.~483--493.

\end{thebibliography}

\end{document}